\documentclass[twoside,a4paper,reqno,11pt]{amsart}
\usepackage{amsfonts, amsbsy, amsmath, amsthm, amssymb, latexsym, verbatim, enumerate}
\usepackage{mathrsfs}
\usepackage[top=30mm,right=30mm,bottom=30mm,left=30mm]{geometry}

\usepackage{hyperref,color}
\usepackage{amssymb,amsfonts,amsmath,amsopn,amstext,amscd,latexsym,amsthm}
\usepackage{bbm}
\usepackage[all,cmtip]{xy}

\usepackage{enumerate}
\usepackage[shortlabels]{enumitem}

\theoremstyle{plain}

\begin{document}

\def\a{\alpha}
 \def\b{\beta}
 \def\e{\epsilon}
 \def\d{\delta}
  \def\D{\Delta}
 \def\c{\chi}
 \def\k{\kappa}
 \def\g{\gamma}
 \def\Ind{\mathrm{Ind}}
 \def\t{\tau}
\def\ti{\tilde}
 \def\N{\mathbb N}
 \def\Q{\mathbb Q}
 \def\Z{\mathbb Z}
 \def\C{\mathbb C}
 \def\F{\mathbb F}
 \def\ovF{\overline\F}
 \def\bfN{\mathbf N}
 \def\bfC{\mathbf C}
 \def\bfO{\mathbf O}
 \def\bfZ{\mathbf Z}
 \def\cG{\mathcal G}
 \def\cT{\mathcal T}
 \def\cX{\mathcal X}
 \def\cY{\mathcal Y}
 \def\cC{\mathcal C}
 \def\cD{\mathcal D}
 \def\cZ{\mathcal Z}
 \def\cO{\mathcal O}
 \def\cW{\mathcal W}
 \def\cL{\mathcal L}
 \def\bfC{\mathbf C}
 \def\bfZ{\mathbf Z}
 \def\bfO{\mathbf O}
 \def\G{\Gamma}
 \def\bO{\boldsymbol{\Omega}}
 \def\bgo{\boldsymbol{\omega}}
 \def\go{\rightarrow}
 \def\do{\downarrow}
 \def\ra{\rangle}
 \def\la{\langle}
 \def\fix{{\rm fix}}
 \def\ind{{\rm ind}}
 \def\rfix{{\rm rfix}}
 \def\diam{{\rm diam}}
 \def\uni{{\rm uni}}
 \def\diag{{\rm diag}}
 \def\Irr{{\rm Irr}}
 \def\Syl{{\rm Syl}}
 \def\Out{{\rm Out}}
 \def\Tr{{\rm Tr}}
 \def\M{{\cal M}}
 \def\cE{{\mathcal E}}
\def\td{\tilde\delta}
\def\tx{\tilde\xi}
\def\DC{D^\circ}
\def\ext{{\rm Ext}}
\def\res{{\rm Res}}
\def\Ker{{\rm Ker}}
\def\hom{{\rm Hom}}
\def\End{{\rm End}}
 \def\rank{{\rm rank}}
 \def\soc{{\rm soc}}
 \def\Cl{{\rm Cl}}
 \def\A{{\sf A}}
 \def\sP{{\sf P}}
 \def\sQ{{\sf Q}}
 \def\SSS{{\sf S}}
  \def\SQ{{\SSS^2}}
 \def\St{{\sf {St}}}
 \def\p{\ell}
 \def\ps{\ell^*}
 \def\SC{{\rm sc}}
 \def\supp{{\sf{supp}}}
  \def\cR{{\mathcal R}}
 \newcommand{\tw}[1]{{}^#1}

\def\Der{{\rm Der}}
 \def\Sym{{\rm Sym}}
 \def\PSL{{\rm PSL}}
 \def\SL{{\rm SL}}
 \def\Sp{{\rm Sp}}
 \def\GL{{\rm GL}}
 \def\SU{{\rm SU}}
 \def\GU{{\rm GU}}
 \def\SO{{\rm SO}}
 \def\PO{{\rm P}\Omega}
 \def\Spin{{\rm Spin}}
 \def\PSp{{\rm PSp}}
 \def\PSU{{\rm PSU}}
 \def\PGL{{\rm PGL}}
 \def\PGU{{\rm PGU}}
 \def\Iso{{\rm Iso}}
 \def\Stab{{\rm Stab}}
 \def\GO{{\rm GO}}
 \def\Ext{{\rm Ext}}
 \def\E{{\cal E}}
 \def\l{\lambda}
 \def\ve{\varepsilon}
 \def\Lie{\rm Lie}
 \def\s{\sigma}
 \def\O{\Omega}
 \def\o{\omega}
 \def\ot{\otimes}
 \def\op{\oplus}
 \def\oc{\overline{\chi}}
 \newcommand{\edit}[1]{{\color{red} #1}}
 \def\pf{\noindent {\bf Proof.$\;$ }}
 \def\Proof{{\it Proof. }$\;\;$}
 \def\no{\noindent}
\def\hal{\unskip\nobreak\hfil\penalty50\hskip10pt\hbox{}\nobreak
 \hfill\vrule height 5pt width 6pt depth 1pt\par\vskip 2mm}

 \renewcommand{\thefootnote}{}

\newtheorem{theorem}{Theorem}
 \newtheorem{thm}{Theorem}[section]
 \newtheorem{prop}[thm]{Proposition}
 \newtheorem{conj}[thm]{Conjecture}
 \newtheorem{question}[thm]{Question}
 \newtheorem{lem}[thm]{Lemma}
 \newtheorem{lemma}[thm]{Lemma}
 \newtheorem{defn}[thm]{Definition}
 \newtheorem{cor}[thm]{Corollary}
 \newtheorem{coroll}[theorem]{Corollary}
\newtheorem*{corB}{Corollary}
 \newtheorem{rem}[thm]{Remark}
 \newtheorem{exa}[thm]{Example}
 \newtheorem{cla}[thm]{Claim}

\numberwithin{equation}{section}
\parskip 1mm

\title{Sectional rank and Cohomology II}

\author[Guralnick]{Robert M. Guralnick}
\address{R.M. Guralnick, Department of Mathematics, University of Southern California, Los Angeles,
CA 90089-2532, USA}
\email{guralnic@usc.edu}
 
\author[Tiep]{Pham Huu Tiep}
\address{P. H. Tiep, Department of Mathematics, Rutgers University, Piscataway, NJ 08854, USA}
\email{tiep@math.rutgers.edu}

 
\date{\today}

\thanks{The first author was partially supported by the NSF
grant DMS-1901595.
The second author was partially supported by the NSF grant DMS-1840702 and the Joshua Barlaz Chair in Mathematics.}


\begin{abstract}
We investigate bounds of $\dim H^m(G,V)$ for $G$ a finite group
of bounded sectional $p$-rank and $V$ irreducible $G$-module over an algebraically closed
field of characteristic $p$.
\end{abstract}

\maketitle



\section{Introduction}

Let $G$ be a finite group, $p$ a prime and $k$ an algebraically closed field of characteristic $p$.
Donovan's conjecture (cf. \cite{K}) asserts that for a fixed $p$-group $D$, 
there are only finitely many
blocks $B$  of any group algebra $kG$ with defect group $D$ up to Morita equivalence.

A trivial consequence of this conjecture is that there is
a bound on the  dimension of $\Ext$-groups between irreducible
modules (depending only on the defect group of the block
containing the irreducibles and the degree).   We investigate bounds
under a weaker hypothesis. 

Recall that the {\it sectional $p$-rank} $s_p(G)$ of a finite group $G$ is the maximal rank 
of an elementary abelian group isomorphic to $L/K$ for some subgroups $K \lhd L$ of $G$.
In this paper we obtain results in cohomology that depend only the sectional $p$-rank of $G$
and improve the results from \cite{GT2}. 
See also \cite{MR} for some recent results about the number of simple modules in a block
in terms of the sectional $p$-rank of the defect group.

 More precisely we obtain some bounds on $\dim \Ext_G^m(D,V)$ under various assumptions, 
 for any irreducible $kG$-module $V$ and for any $kG$-module $D$.
The best results are for $m=1$ or when $G$ is $p$-solvable.   Our arguments and results are an improvement on those
from \cite{GT2}.  Also, we critically use a recent result of Symonds \cite{sym} that implies that there is
a bound on $\dim H^m(G,k)$ depending only on $m$ and $s_p(G)$.    We will use this via
the following:

\begin{thm} \label{symonds} {\rm \cite{sym}}  Let $p$ be a prime, $k$ an algebraically closed field of
characteristic $p$, and let $m$ and $s$ be any natural numbers.
Then there is a constant $C=C(m, s)$, such that $\dim H^m(G,V) \le C(\dim V)$
for any finite group $G$ with 
$s_p(G) \le s$ and  $kG$-module $V$.
\end{thm}

Symonds \cite{sym}  proved this result for $G$ a $p$-group and $V=k$.
By restricting to a Sylow $p$-subgroup, it is clear that this implies the statement above.  In fact, Symonds gives an explicit bound on $C$ (which does not depend on $p$).   

Let us call a finite group $G$  {\it weakly $p$-solvable} if every composition factor of $G$ has cyclic Sylow $p$-subgroups.

Our first main result is the following:

\begin{thm} \label{main:p-solvable}  
There exists a function $f: \N^3 \to \N$ such that the following statement holds.
Let $p$ be a prime, $k$ an algebraically closed field of
characteristic $p$, $m$ and $s$ natural numbers and $G$ a finite weakly $p$-solvable group with 
$s_p(G) \le s$.    Let $D$ and $V$ be  $kG$-modules with $V$ irreducible and $\dim D \le d$.
Then there is a constant $C \leq f(m, s, d)$
such that $\dim \Ext_G^m(D,V) \le C$.
\end{thm}

Note that the bound $f(m,s,d)$ in the previous theorem does not depend on $p$.
We also prove the following Theorem \ref{thm:reduction} which reduces the problem  in general to obtaining
bounds for cohomology for simple groups of bounded sectional $p$-rank.   There is such
a bound (depending upon $p$) for $H^1$, see \cite[Theorem 1.4]{GT} which combine the main results of \cite{CPS} and \cite{GT} 
(we conjecture that there exists an upper bound that does not depend on $p$).
 
We first extend \cite[Corollary 1.2(ii)]{GT2}.

\begin{thm} \label{main:ext1}  
There exists a function $g: \N^3 \to \N$ such that the following statement holds.
Let $p$ be a prime, $k$ an algebraically closed field of
characteristic $p$ and $G$ a finite group with 
$s_p(G) \le s$.    Let $D$ and $V$ be  $kG$-modules with $V$ irreducible and $\dim D \le d$.
Then there is a constant $C\leq g(p, s, d)$ such that $\dim \Ext_G^1(D,V) \le C$.
\end{thm}

If we consider higher degree, we show that theorem reduces to the case of simple groups.  

\begin{thm}  \label{thm:reduction}  Let $p$ be a prime, $k$ an algebraically closed field of
characteristic $p$, and let $m$, $s$ be positive integers.
Assume there exists a constant $C=C(p,s,m)$ such that 
$\dim H^j(S,W) \leq C$ for any $1 \leq j \le m$, any finite simple group $S$ of sectional $p$-rank at most 
$s$, and any irreducible $kS$-module $W$.
Then there exists a function $h:\N^4 \to \N$ such that the following statement holds.
For any finite group $G$ with $s_p(G) \le s$, any $d \in \N$, and    
any $kG$-modules $V$ and $D$ with $V$ irreducible and $\dim D \le d$, we have 
$$\dim \Ext_G^m(D,V) \le h(p,s,d,m).$$
 \end{thm}    
 
 The proof of the previous theorem shows that if the bound $C(p,s,m)$ we obtain for simple groups is independent of
 $p$, then the same is true for general $G$, that is, one can choose $h(p,s,d,m)$ to be independent of $p$.  
 
 Note that if $G$ is weakly $p$-solvable, then, one can choose the bound $C(p,s,m)$ to be $1$
 for simple groups with cyclic Sylow $p$-subgroups,
 and so Theorem \ref{main:p-solvable} is a consequence of Theorem \ref{thm:reduction}.  However,  a version
 of Theorem \ref{main:p-solvable} is used in the proof of Theorem \ref{thm:reduction}.
 
Furthermore, if $m=1$, such a bound $C(p,s,1)$ for simple groups 
 exists, see \cite[Theorem 1.4]{GT}.   However, it is still
 unknown (but rather unlikely) whether there exists an absolute constant $B$ (independent from $p$ and $s$) 
 so that $\dim H^1(S,V) < B$ for $S$ simple
 and $V$ irreducible (or equivalently if $S$ is any finite group and $V$ is irreducible and faithful).   See
 \cite{Gu1, Gu2, GH}.   

There is even less known for $m > 1$.  
 It is known \cite{CPS}  that for
$G$ a finite simple group of Lie type of bounded rank $s$ with $k$
the same characteristic $p$ as of $G$,  there is a bound
$\dim H^m(G,V) < C_1(p,s,m)$ for $V$ any absolutely irreducible $kG$-module.

\section{Preliminaries}

We use the standard notation as in \cite{Asbook}.  See \cite{Brown} for standard facts about cohomology.

We will use the following well known properties of cohomology. 
We fix a prime $p$ and an algebraically closed field $k$ of characteristic $p$. 
We first note the trivial result:

\begin{lemma} \label{trivial} Let $H$ be a finite group and $U=W_1 \otimes W_2$
a tensor product of $kH$-modules with $W_2$ irreducible.  Then 
$\dim H^0(H,U) \le (\dim W_1)/(\dim W_2)$.
\end{lemma}

Applying Lemma \ref{trivial} to $U = D^* \otimes_k V$, this proves Theorem \ref{thm:reduction} for $\Ext^0$. 

We need the following result that follows from an easy spectral sequence
argument.  See \cite{Ho} or \cite[Lemma 3.7]{GKKL}. 

\begin{lemma} \label{holt}   Let $G$ be a finite group, $N$ a normal
subgroup of $G$ and $V$ a $kG$-module.  Then 
$\dim H^n(G,V) \le \sum_{i=0}^n  \dim H^i(G/N, H^{n-i}(N,V))$.
\end{lemma}

The next result   follows by taking $kN$-resolutions for $V_1$ and tensoring
with $V_2$ (see also the proof of Lemma 2.4 in \cite{GT2}).

\begin{lemma} \label{tensor}  Let $G$ be a finite group with a normal subgroup $N$.
Let $V$ be an irreducible $kG$-module such that $V=V_1 \otimes V_2$ where 
the $V_i$ are irreducible, $N$ acts irreducibly on $V_1$ and $N$ acts trivially on $V_2$.   Then as a $G$-module, 
$H^m(N,V)) \cong H^m(N,V_1) \otimes V_2$.  
\end{lemma}

We next show that the existence of a bound for the cohomology groups for simple groups
implies a bound for $\Ext$ groups for quasisimple groups.

\begin{lemma} \label{coh2ext}  There exists a function $f:\N^2 \to \N$ such that the following statement holds.
Let $G$ be a quasisimple group of sectional rank $p$-rank $s$,  $V$ an irreducible
$kG-module$ and $D$ a $kG$-module of dimension $d$.     Then 
$\dim \Ext^m_G(D,V) \le f(d,s) \dim H^m(G,V)$.
\end{lemma} 

\begin{proof} 
Note that, up to isomorphism, there exists a bounded in $s$ number of quasisimple groups whose non-abelian 
composition factor is either a sporadic group, or a Lie type in characteristic $p$. Trivially, 
there exists a bounded in $d$ quasisimple covers of an alternating group ${\mathsf {A}}_n$ that admit a 
nontrivial representation of degree $\leq d$ -- indeed, under such assumptions we must have $n \leq \max(d+2,9)$.
Likewise, the Landazuri-Seitz-Zalesskii bound \cite{LS} implies that there exists a bounded in $d$ quasisimple covers of a simple group 
of Lie type in characteristic $\neq p$ that admit a 
nontrivial representation over $k$ of degree $\leq d$.
Now, with $d$ and $s$ fixed, we may exclude a bounded in $d,s$ number of quasisimple groups, and therefore assume that $G$ is a 
quasisimple group of Lie type
in some characteristic $r \ne p$ and that $d$ is less than the dimension of any nontrivial irreducible $kG$-module.
In that case $G$  acts trivially on $D$, and so $\dim \Ext^m_G(D,V) = d \cdot \dim H^m(G,V)$.
\end{proof}  

Next note that a special case of Lemma \ref{holt} is the following:

\begin{lemma} \label{p'cohomology}   Let  $G$ be a finite group, $V$ a finite dimensional $kG$-module
and $N$ a normal $p'$-subgroup of $G$.  If $H^0(N,V)=0$, then $H^j(G,V)=0$ for all $j$.
\end{lemma}

We will also need the following result.   If $p$ is a prime, let $R_p(G)$ denote the maximal $p$-solvable normal
subgroup of $G$ and $R(G)$ the solvable radical of $G$, $F(G)$ the Fitting subgroup of $G$, and $E(G)$ the {\it layer} of $G$ (that is, the product of components of $G$), so that $F^*(G)=E(G)F(G)$.

\begin{lemma}  \label{boundedmode}  Let $G$ be a finite group with $\bfO_p(G)=1$ and $\bfO_{p'}(G)=\bfZ(G)$.   
Assume that the number of components of $G$ is at most $t$.   If $H:=G/E(G)$,  then 
$|H/R(H)|$ divides $t!$. 
\end{lemma} 

\begin{proof}  
By the assumptions, $F(G) = \bfZ(G)$. Hence $G$ acts on $E(G)$, with centralizer $\bfC_G(E(G))=\bfC_G(F(G)E(G)) \leq \bfZ(G)E(G)$. 
It follows that $G/\bfZ(G)E(G)$ embeds into the outer automorphism group of $E(G)$, which permutes the $t$ components in $E(G)$.
By the Schreier conjecture, the kernel of the latter permutation action is solvable, and so its inverse image $K$ in $H$ is contained in 
$R(H)$. As $H/K$ embeds in ${\mathsf {Sym}}_t$, the result follows. 
\end{proof}

\section{Some Reductions}

Fix a prime $p$ and an algebraically closed field $k$ of characteristic $p$.   All modules for a finite group
$G$ will be (finite-dimensional) $kG$-modules. 

We will prove Theorem \ref{thm:reduction} by double induction, first on $m$ and then on $d$  and also by proving it in some special cases first.
Thus, we assume that the theorem holds for any smaller than $m$ cohomology degree (for arbitrary $d$) and in the given degree $m$
for any dimension $< d$, and also that it holds in the previous special cases.

The special families of groups we consider are (still with the assumption of bounded sectional $p$-rank):
\begin{enumerate}[\rm(A)]
\item $|G|$ is bounded;  
\item  $G$ is $p$-solvable; and
\item  $|G/R(G)|$ is bounded.    
\end{enumerate}

These three cases do not depend on the corresponding result for simple groups and so the
statement of the theorem in these cases is unconditional.  The first two cases are used in the proof
of the third case, and (C) is used in the general case. 

We first note if $|G| \leq M$ is bounded, the result is clear:   we may assume that $D$ is irreducible
and so there only finitely many possibilities for $G$, $D$ and $V$, whence there is a bound
depending only on $m$ and $M$.    

There are a few global reductions that we use. 

\noindent
{\it Step 1:  We may assume that $D$ is irreducible and $V$ is primitive}. 

\begin{proof} 
The first claim follows from the sub-additivity of the dimension of cohomology groups. 
Next, if $V$ is imprimitive, then there $V=\Ind_H^G(W)$ for some subgroup $H$ with $s_p(H) \leq s_p(G) \leq s$ and a 
primitive irreducible $kH$-module $W$.
Then by Shapiro's Lemma, $\dim \Ext_G^m(D,V) = \dim \Ext_H^m(D,W)$, and the result holds. 
\end{proof} 

\noindent 
{\it Step 2:  We may assume that if $K$ is the kernel of  $G$ on $V$, then $K/\bfO_p(G)$ is a
$p'$-group.}  

\begin{proof} Let $Q$ be a Sylow $p$-subgroup of $K$.   By the Frattini argument, $G=\bfN_G(Q)K$ and so
$H:=\bfN_G(Q)$ acts irreducibly (and primitively, by previous step) on $V$.    Since
$H$ contains a Sylow $p$-subgroup of $G$,  $\dim \Ext_G^m(D,V) \le \dim \Ext_H^m(D,V)$.
Thus, we may assume that $Q$ is normal in $G$. But $\bfO_p(G) \leq K$, so we conclude $Q=\bfO_p(G)$ and so $K/\bfO_p(G)$ 
is a $p'$-group.
\end{proof} 

\noindent
{\it Step 3:   We may assume that the kernel $K$ of $G$ on $V$ is a $p'$-group.} 

\begin{proof}  
By the previous step, $K/\bfO_p(G)$ is a $p'$-group.   By 
Lemma \ref{holt} and Lemma \ref{tensor} we see that 
\begin{align}
\dim H^m(G, D^* \otimes V)  & \le \sum_{j=0}^m \dim H^j(G/\bfO_p(G), H^{m-j}(\bfO_p(G), D^* \otimes V))\\
& =  \sum_{j=0}^m \dim H^j(G/\bfO_p(G), H^{m-j}(\bfO_p(G), k) \otimes D^* \otimes V)).
\end{align} 
\noindent
The term with $j=m$ is $H^m(G/\bfO_p(G), D^* \otimes V)$.   For all the other terms we obtain a
bound by induction and the result follows. 
\end{proof} 

\noindent 
{\it Step 4:  Let $K$ be the kernel of $G$ on $V$ (which we may assume by the previous step to be a $p'$-group).
We may assume that there is a $1$-dimensional representation $\rho$ of $\bfO_{p'}(G)$
such that $\bfO_{p'}(G)$ acts via $\rho$ on both $V$ and $D$.   In particular,  $\bfO_{p'}(G)/K = \bfZ(G/K)$.}

\begin{proof}  
Since the irreducible $G$-module $V$ is primitive by Step 1, the restriction of $V$ to $\bfO_{p'}(G)$ is homogeneous, i.e. 
a direct sum of copies of the same irreducible $\bfO_{p'}(G)$-module $U$. 
If $U$ is not a summand of $D|_{\bfO_{p'}(G)}$, then
$\bfO_{p'}(G)$ has no fixed points on $D^* \otimes V$ and so all the cohomology groups are trivial, and the result follows.
In the remaining case, since $U$ is $G$-invariant and $D$ is irreducible, it follows that $D|_{\bfO_{p'}(G)}$ is also
a direct sum of copies of $U$.
Passing to a universal $p'$-cover of $G$, which does not change its sectional $p$-rank, we may write 
$V \cong W \otimes X$, where $W$ is an irreducible $kG$-module
$W|_{\bfO_{p'}(G)} \cong U$ and $X$ is an irreducible $kG$-module on which 
$\bfO_{p'}(G)$ acts trivially (see e.g. \cite[Theorem (51.7)]{CR}).  In turn, with $W$ fixed, we can also write   
write $D \cong W \otimes Y$ where $Y$ is an irreducible $kG$-module on which  
$\bfO_{p'}(G)$ acts trivially. 

Now we have  
$$H^j(G,D^* \otimes V)
= H^j(G,  (W^* \otimes W) \otimes Y^* \otimes X).$$  
Since $W|_{\bfO_{p'}(G)}$ is irreducible, we can write 
$W^* \otimes W = k \oplus  U$, where $\bfO_{p'}(G)$ has acts fixed-point-freely on $U$, and this decomposition is $G$-equivariant. 
As $\bfO_{p'}(G)$ acts trivially on both $X$ and $Y$,  it has no fixed points on $U \otimes Y^* \otimes X$,
whence all $H^j(G, U \otimes Y^* \otimes X)=0$ and so $H^j(G, D^* \otimes V) \cong H^j(G, Y^* \otimes X)$.

If $\dim W > 1$, the result follows by induction on $d$.  If $\dim W = 1$, then the conclusion holds
and since $W$ extends to a $kG$-module, the last statement holds as well. 
\end{proof}  

\section{Completion of the Proof}

We keep notation as in the previous section. 

Let us first note that the result for $p$-solvable groups holds by the reductions in the previous case.
We have reduced to the case where $V$ is faithful and $\bfO_{p'}(G)$ is central.  If $G$ is $p$-solvable,
this implies that $F^*(G)=\bfZ(G)$ is cyclic. Hence, $G = \bfC_{G}(F^*(G)) \leq F^*(G)$ is also is cyclic, and the result
follows in case (B). 

\smallskip
Next we show that the result holds in case (C) where $|G/R(G)|$ is bounded. Indeed, by Lemma \ref{holt}, 
$$
\dim H^m(G, D^* \otimes V)  \le \sum_{j=0}^m \dim H^j(G/R(G), H^{m-j}(R(G), D^* \otimes V)).
$$
Now, Lemma \ref{tensor} and the results for bounded groups  (applied to $G/R(G)$) and solvable groups (applied to
$R(G)$) imply the statement for $G$. 

\smallskip
We next prove the theorem in the case that the kernel of $G$ on $V$ is a $p'$-group
and $G = E(G)\bfZ(G)$ under the assumption that we have bounds on $\dim H^j(S,W)$
for $S$ a simple group of bounded sectional $p$-rank and $W$ an irreducible $kS$-module.

\begin{thm}\label{layer}  
Suppose there exists a constant $C_1=C_1(p, j, s)$
such that 
$$\dim H^j(S,W) \le  C_1$$ 
for every finite simple group $S$ of sectional $p$-rank at most $s$.
Then there exists a constant $C_2=C_2(p,j,d,s)$ such that the following statement holds.  
Let $G=E(G)\bfZ(G)$ be any finite group of sectional $p$-rank at most $s$ and with
$\bfZ(G)$ being a $p'$-group.    Let $V$ be any irreducible $kG$-module and let 
$D$ be a $kG$-module of dimension at most $d$. Then 
$$\dim \Ext_G^j(D,V) \leq C_2.$$  
\end{thm}

\begin{proof}  By passing to a central $p'$-extension, we may assume that $G$ is a direct product of
quasisimple groups and a central subgroup.   By Lemma \ref{coh2ext}, the result holds for 
quasisimple groups and then for $G$ by the K\"unneth formula.
\end{proof}

The general case in Theorem \ref{thm:reduction} now follows by the same argument. By the reductions in \S3 may assume that 
$G$ acts faithfully on $V$ and $\bfO_{p'}(G)$ is central and acts the same on both $D$ and $V$.  
Let $N=E(G)\bfO_{p'}(K)$.   Then 
$$
\dim H^m(G, D^* \otimes V)  \le \sum_{j=0}^m \dim H^j(G/N, H^{m-j}(N, D^* \otimes V)).
$$

By Lemma \ref{tensor}  as a $G$-module, $H^{m-j}(N, D^* \otimes V) = H^{m-j}(N, D^* \otimes W) \otimes X$.
By Theorem \ref{layer},  there is a bound on $\dim H^{m-j}(N, D^* \otimes W)$ and then we can apply
the result for $G/N$ (which is bounded modulo a solvable normal subgroup by Lemma \ref{boundedmode}).

\end{document}